\documentclass[bj,numbers]{imsart}

\RequirePackage{amsthm,amsmath,amsfonts,amssymb}
\RequirePackage[numbers]{natbib}
\usepackage[colorlinks,citecolor=blue,urlcolor=blue]{hyperref}
\RequirePackage{graphicx}
\usepackage{subfigure}
\usepackage{multirow}
\usepackage{hhline}
\usepackage{booktabs}
\usepackage{float}
\usepackage{enumitem}
\setlist[itemize]{label=\textbullet, leftmargin=*}
\startlocaldefs
\theoremstyle{plain}

\newtheorem{theorem}{Theorem}[section]
\newtheorem{lemma}[theorem]{Lemma}
\newtheorem{corollary}[theorem]{Corollary}

\theoremstyle{remark}

\newtheorem{remark}[theorem]{Remark}
\newtheorem{assumption}{Assumption}

\newcommand{\bbA}{{\bf A}}

\newcommand{\bbC}{{\bf C}}
\newcommand{\bbD}{{\bf D}}
\newcommand{\bbS}{{\bf S}}

\newcommand{\bbT}{{\bf T}}
\newcommand{\bbI}{{\bf I}}

\newcommand{\bbB}{{\bf B}}
\newcommand{\bbX}{{\bf X}}
\newcommand{\bbx}{{\bf x}}

\newcommand{\bbY}{{\bf Y}}

\newcommand{\bbZ}{{\bf Z}}

\newcommand{\bbe}{{\bf e}}

\newcommand{\bqn}{\begin{eqnarray*}}
	\newcommand{\eqn}{\end{eqnarray*}}

\newcommand{\rE}{{\textrm{E}}}

\newcommand{\bbf}{{\bf f}}

\newcommand{\bbw}{{\bf w}}

\newcommand{\bbU}{{\bf U}}
\newcommand{\bbu}{{\bf u}}
\newcommand{\bbV}{{\bf V}}
\newcommand{\bbv}{{\bf v}}

\numberwithin{equation}{section}

\endlocaldefs

\begin{document}

\begin{frontmatter}
\title{The Asymptotic Properties of the Extreme Eigenvectors of High-dimensional Generalized Spiked Covariance Model}
\runtitle{The Asymptotic Properties of the Extreme Eigenvectors}

\begin{aug}
\author[A]{\inits{F.}\fnms{Zhangni}~\snm{Pu}\ead[label=e1]{puzn687@nenu.edu.cn}}
\author[A]{\inits{S.}\fnms{Xiaozhuo}~\snm{Zhang}\ead[label=e2]{zhangxz722@nenu.edu.cn}\orcid{0009-0000-6333-7768}}
\author[A]{\inits{T.}\fnms{Jiang}~\snm{Hu}\ead[label=e3]{huj156@nenu.edu.cn}}
\author[A]{\inits{F.}\fnms{Zhidong}~\snm{Bai}\ead[label=e4]{baizd@nenu.edu.cn}
\orcid{0000-0002-5300-5513}}
\address[A]{KLASMOE and School of Mathematics $\&$ Statistics, Northeast Normal University, No. 5268 People's Street, Changchun 130024, China\printead[presep={,\ }]{e1,e2,e3,e4}}

\end{aug}

\begin{abstract}
In this paper, we investigate the asymptotic behaviors of the extreme eigenvectors in a general spiked covariance matrix, where the dimension and sample size increase proportionally. We eliminate the restrictive assumption of the block diagonal structure in the population covariance matrix. Moreover, there is no requirement for the spiked eigenvalues and the 4th moment to be bounded. Specifically, we apply random matrix theory to derive the convergence and limiting distributions of certain projections of the extreme eigenvectors in a large sample covariance matrix within a generalized spiked population model. Furthermore, our techniques are robust and effective, even when spiked eigenvalues differ significantly in magnitude from nonspiked ones. Finally, we propose a powerful statistic for hypothesis testing for the eigenspaces of covariance matrices.
\end{abstract}


\end{frontmatter}

\section{Introduction}

In statistics, the sample covariance matrix describes the relationships between variables in a dataset. Based on the theory of random matrices, researchers have traditionally investigated their asymptotic properties in high-dimensional contexts where  both the sample size and dimension increase. Specific first and second-order results for the eigenvalues and eigenvectors of sample covariance matrices have been discussed by Bai and Silverstein in \cite{2009baisil}, Couillet and Debbah in \cite{2011CD}, and Yao, Zheng and Bai in \cite{2015YZB}.

In recent years, the concept of the spiked sample covariance matrix has become popular and significant in modern statistical theory and practice, such as in wireless communications \cite{2011CD}, fault detection in sensor networks \cite{2013CH}, and signal detection \cite{2009B}, where a few variables may have a much stronger influence on the overall data structure. The spiked sample covariance matrix is a modified version of the sample covariance matrix designed to accommodate the presence of a few dominant factors, known as ``spikes,'' in the data. These spikes correspond to a few large or small eigenvalues of the population covariance matrix that are assumed to be well separated from the rest. By incorporating information about these spikes, the spiked sample covariance matrix can improve the accuracy of statistical inference and reduce the risk of false discoveries. In many cases, it serves as a more accurate and useful estimator of the true population covariance matrix. 

Consider a sequence of independently and identically distributed (i.i.d.) $p$-dimensional real random vectors $(\bbx_1,\cdots,\bbx_n)=\bbX=(x_{ij}), 1\leqslant i\leqslant p,1\leqslant j\leqslant n$, with zero mean and population covariance matrix $\pmb{\Sigma}$. Define the corresponding sample covariance matrix as 
\begin{equation}\label{model}
	\bbB=\frac{1}{n}\bbT\bbX\bbX^*\bbT^*,
\end{equation}
with $l_1\geq l_2\geq \cdots \geq 0$ denoting the eigenvalues of $\bbB$ with corresponding eigenvector $\bbf_i$.
It is well known that the empirical spectral distribution (ESD) of a large sample covariance matrix converges to the family of Mar\'{c}enko-Pastur laws under fairly general conditions on the sample variables \cite{1967mplaw}. For example, when the sample vectors have independent coordinates and unit variances and assuming that the ratio $p/n$ tends to a positive limit $y\in(0,1)$, then the limiting distribution is the classical Mar\'{c}enko-Pastur law $F_y(dx)$
\begin{equation}
	F_y(dx)=\dfrac{1}{2\pi xy}\sqrt{(x-a_y)(b_y-x)}dx,\quad a_y\leq x\leq b_y,
\end{equation}
where $a_y=(1-\sqrt{y})^2$, and $b_y=(1+\sqrt{y})^2$. Furthermore, the smallest and largest eigenvalues converge almost surely to the boundary $a_y$ and $b_y$, respectively.

In the context of Random Matrix Theory (RMT), the spiked model was first studied by Johnstone in \cite{2001Johnstone}. The spiked model is primarily focused on the asymptotic behaviors of a few largest eigenvalues and their associated eigenvectors in a non-null case where all eigenvalues of $\pmb{\Sigma}$ are equal to one except for a fixed small number of spikes, i.e.,
\begin{equation}
	\text{Spec}\left( \pmb{\Sigma}\right)=\{ \sigma_1, \cdots, \sigma_M, \underbrace{1, \cdots, 1}_{p-M}\} . 
\end{equation} 
Additionally, in the null case where $\pmb{\Sigma}$ is the identity $\bbI$, Johnstone \cite{2001Johnstone} established the Tracy-Widom law for the maximum eigenvalues of a real Wishart matrix $\bbB$.
Consequently, high-dimensional statistical inference requires a comprehensive understanding of the distributions of eigenvalues and eigenvectors of sample covariance matrices. In Johnstone's spiked model settings, a systematic examination of individual eigenvalues has been extensively explored by Baik, Ben Arous and P\'{e}ch\'{e} in \cite{2005BBP}. This paper investigates the limiting behavior of the largest eigenvalue in the context of complex Gaussian variables. A similar phase transition phenomenon of largest sample eigenvalues was shown in Baik and Silverstein \cite{2006baiksil}. They extensively investigate the almost sure limits of sample eigenvalues under the Mar\'{c}enko-Pastur regime as $p,n\to\infty$, $p/n\to y\in(0,\infty)$, and found out that these limits vary depending on whether the corresponding population spiked eigenvalues are larger or smaller than the critical values $1+\sqrt{y}$ and $1-\sqrt{y}$. Subsequently, extensive research has been dedicated to investigating the asymptotic properties of spiked eigenvalues in high-dimensional covariance matrices, including contributions from Paul \cite{2007paul} and Bai and Yao \cite{2008baiyao}. To improve the simplified assumptions, Bai and Yao \cite{2012baiyao} contributed to dealing with a more general spiked model, i.e.,
 \begin{equation}
	\pmb{\Sigma}=\begin{pmatrix} \pmb{\Lambda} & 0 \\ 0 & \bbV \end{pmatrix},
\end{equation} 
in which a condition of the diagonal block $\pmb\Sigma$ independence and finite 4th moments are assumed. Building on these studies, Jiang and Bai \cite{PG4MT} further explore a general spiked covariance matrix by applying the proposed partial Generalized Four Moment Theorem (PG4MT) to the Central Limit Theorem (CLT) for its spiked eigenvalues. They give the explicit CLT for the spiked eigenvalues of high-dimensional generalized covariance matrices. Furthermore, efforts have been made to enhance research on the spiked population model using Principal Component Analysis (PCA) or Factor Analysis (FA). For instance, Bai and Ng \cite{2002baiNg}, Hoyle and Rattray \cite{2004hoyle}, Onatski \cite{2009onatski}, and so on. 

In contrast, research on the asymptotic properties of eigenvectors of large dimensional random matrices is generally more challenging and less developed. However, eigenvectors play an important role in high-dimensional statistics and have provided valuable insights into the behavior of complex systems. Understanding the asymptotic properties of eigenvectors corresponding to spiked eigenvalues enables researchers to develop novel techniques for analyzing and manipulating high-dimensional data, modeling complex systems, and solving a variety of significant computational problems. 

The earlier investigations into the properties of eigenvectors trace back to Anderson \cite{1963anderson}, where the author established the asymptotic normality of eigenvectors of the Wishart matrix as $M$ is fixed and $n$ tends to infinity. For the high-dimensional case, Johnstone \cite{2001Johnstone} introduced the spiked model to test the existence of principal components. Subsequently, Paul \cite{2007paul} studied the directions of eigenvectors corresponding to spiked eigenvalues. The study focuses on a spiked model structure in the covariance matrix, where there is a large eigenvalue and several smaller eigenvalues. A novel estimation approach is proposed, with its asymptotic normality and consistency of the method are proved by studying the asymptotic behavior of sample eigenvalues and eigenvectors. 

Moreover, the spectral properties of spiked covariance models have been extensively studied, yielding precise analytical results regarding the asymptotic first-order and distributional properties of both eigenvalues and eigenvectors; see, for example, Baik and Silverstein \cite{2006baiksil}; Bai and Yao \cite{2008baiyao}; Benaych-Georges and Nadakuditi \cite{2011BN}; Couillet and Hachem \cite{2013CH}; Bloemendal et al. \cite{2016Betal}. For comprehensive reviews, readers may refer to Couillet and Debbah \cite{2011CD} and Yao, Zheng and Bai \cite{2015YZB}. The more general work is the recent contributions by Bao et al. \cite{2020baoding},  which study the asymptotic behavior of the extreme eigenvalues and eigenvectors of the high-dimensional spiked sample covariance matrices. They require the bounded infinite moments and the condition $\pmb\Sigma=\bbI+\bbS$, where $\bbS$ is a fixed-rank deterministic Hermitian matrix. The limiting behavior of the extreme eigenvectors has also been studied for various related models, such as the finite-rank deformation of Wigner matrices for the extreme eigenvalues \cite{2009ca,2013isotropic,2014deformed}. In \cite{2018ca} and \cite{2020baosingular}, a non-universality phenomenon for the eigenvector distribution for the finite-rank deformation of Wigner matrices and the signal-plus-noise model were demonstrated, respectively. Here we also refer to \cite{2018MJMY,2018J&Y,2019fan} for related study on the extreme eigenstructures from more statistical perspective. 

In this paper, we further consider a general spiked covariance matrix and apply random matrix theory to investigate the asymptotic properties of eigenvectors corresponding to extreme eigenvalues of large sample covariance matrix in a generalized spiked population model. Our analysis primarily relies on tools borrowed from random matrix theory. For a comprehensive understanding of this theory, readers may refer to the seminal book \cite{1991mehta} and a modern review by Bai \cite{1999bai}. Additionally, we introduce another important tool in this paper, namely the PG4MT for deriving a CLT of spiked eigenvalues of sample covariance matrices. The PG4MT is presented in Lemma \ref{PG4MT}.

Our contributions mainly include the following aspects:

\begin{itemize}[leftmargin=4em]
	\item Our work only requires the condition of matching moments up to the 2nd order and a rate $o(x^{-4})$ of the tail probability $P(|X|\geq x)$ as $x\to\infty$. This condition stands as both necessary and sufficient for the weak convergence of the largest eigenvalue.
	\item  We considered a spiked covariance matrix as a general non-negative definite matrix $\pmb{\Sigma}=\bbT\bbT^*$. Here, $\bbT$ undergoes singular value decomposition as $\bbT=\bbV\begin{pmatrix} \bbD_1 & 0 \\ 0 & \bbD_2 \end{pmatrix}\bbU^*$, where $\bbV$ and $\bbU$ are unitary matrices. The entries of $\bbD_1$ encompass the squares of all spiked eigenvalues of $\pmb{\Sigma}$, while those of $\bbD_2$ represent the squares of non-spiked eigenvalues of $\pmb\Sigma$. Automatically, the diagonal block independent assumption given in Bai and Yao \cite{2008baiyao} is removed.
	\item We extend our focus to a generalized scenario where spiked eigenvalues either exceed their related left-thresholds or fall below their related right-thresholds. Importantly, in our work, the spiked eigenvalues are not necessarily required to be bounded.
\end{itemize}

Specifically, in this paper, we consider the generalized spiked sample covariance, represented as (\ref{model}), where $ \bbT $ is a $ p \times p $ deterministic matrix. Consequently, $ \bbT\bbT^*=\pmb\Sigma $ is the population covariance matrix, which can be seen as a general non-negative definite matrix with the spectrum formed as
\begin{equation}\label{dk}
	\sigma_{p,1}, \cdots, \sigma_{p,p}
\end{equation}
in descending order. 

Since the population covariance matrix may have multiple roots, let $\sigma_{p,j_{k}+1}$, $\cdots, \sigma_{p,j_{k}+m_{k}}$ be equal to $d_k^2$, $k=1, \cdots, K$, respectively, where $\mathcal{J}_k=\{j_{k}+1,\cdots,j_{k}+m_{k}\}$ is the set of ranks of the $m_k$-ple eigenvalue
$d_k^2$ in the array (\ref{dk}). Then $d_1^2,\cdots,d_K^2$ with multiplicity $m_k$, $k=1,\cdots,K$, respectively, satisfying $m_1+\cdots+m_K=M$, a fixed integer, are the population spiked eigenvalues of $\pmb\Sigma$ arranged arbitrarily in groups among all the eigenvalues.

To consider the asymptotic properties of the spiked eigenvectors of a generalized sample covariance matrix $\bbB$, it is necessary to determine the following Assumptions \ref{tail}-\ref{distant}: 

\begin{assumption}\label{tail}
	\{$x_{ij}: 1\leq i\leq p, 1\leq j\leq n$\} are independently and identically distributed (i.i.d.) real random variables satisfying $$Ex_{ij}=0,~~~~E|x_{ij}|^2=1$$ and $$\lim\limits_{\tau\to\infty}\tau^4P(|x_{ij}|>\tau)=0.$$

\end{assumption}
\begin{assumption}\label{aT}	
	Denote the empirical spectral distribution (ESD) of $\pmb\Sigma=\bbT\bbT^*$ as $\mathit{H_n}$, which tends to a proper probability measure $\mathit{H}$ as $p\to\infty$.  Moreover, $\bbT$ is nonrandom matrix with $M$, $K$ and $m_k, k=1,\dots,K$ are fixed.
\end{assumption}
\begin{remark}
	$\bbT$ admits the singular decomposition $\bbT=\bbV\begin{pmatrix} \bbD_1 & 0 \\ 0 & \bbD_2 \end{pmatrix}\bbU^*$, where $\bbV$ and $\bbU$ are unitary matrices, $\bbD_1^2=\text{diag}(\underbrace{d_1^2,\dots,d_1^2}_{m_1},\underbrace{d_2^2,\dots,d_2^2}_{m_2},\dots,\underbrace{d_K^2,\dots,d_K^2}_{m_K})$ is a diagonal matrix of the $M$ spiked eigenvalues and $\bbD_2^2$ is the diagonal matrix of the non-spiked eigenvalues with bounded components. Then the  corresponding sample covariance matrix $\bbB=\frac{1}{n}\bbT\bbX\bbX^*\bbT^*$ is the so-called generalized spiked sample covariance matrix.
\end{remark}
\begin{assumption}\label{aU}
	Suppose that all
	\begin{equation}\label{ux}
		\kappa_{x,k}=\lim_{n\to\infty}\left( \sum_{i\in\mathcal{J}_k}\sum_{t=1}^p u_{it}^4+\sum_{\stackrel{i_1\neq i_2}{i_1,i_2\in\mathcal{J}_k}}\sum_{t=1}^p u_{i_1t}^2u_{i_2t}^2\right) \left( Ex_{11}^4\mathit{I}(|x_{11}|<\sqrt{n})-3\right)	
			\end{equation}
	are finite,
		where $\mathit{I(\cdot)}$ is indicator and $\bbU_1=(u_{ts})_{t=1,\cdots,p;s=1,\cdots,M}$ is the first $\mathit{M}$ columns of matrix $\bbU$.
\end{assumption}

\begin{remark}
	In this paper, we consider the case where the 4th moment may unnecessarily exist. By (\ref{ux}), note that $\kappa_{x,k}=0$ for the case $Ex_{11}^4=3$ or the case $\max_{i,t}u_{it}^2\to 0$ and the 4th moment is finite. If the 4th moment does not exist, we only need to take the appropriate $u_{it}^2$ to ensure that $\kappa_{x,k}$ is finite.

\end{remark}
\begin{assumption}\label{p/n}	
	$p/n=c_n\to c>0$ and both $n$ and $p$ go to infinity simultaneously. 
\end{assumption}
\begin{assumption}\label{distant}
	For any $1\leq j,k \leq K$, there exists some constant $d>0$ independent of $p$ and $n$, such that
	\begin{equation}\label{sc}
		\min\limits_{j\neq k}\left| \dfrac{d_k^2}{d_j^2}-1\right| >d.
	\end{equation}
 Moreover, for $1\leq k\leq K$, $d_k^2$ satisfy ${\psi\textquotesingle(d_k^2)}>0$, where
\begin{equation}\label{psi}
\psi(x)=x\left( 1+c\int\dfrac{t}{x-t}\text{d}H(t)\right)
\end{equation}
as detailed in Proposition 2.1 in Jiang and Bai \cite{2021jiangbai}.	
\end{assumption}
\begin{remark}
	Assumption \ref{distant} ensures that the spiked eigenvalues of the matrix $\pmb\Sigma$, $d_1^2,\cdots,d_K^2$ with multiplicities $m_1,\cdots,m_K$ laying outside the support of $\mathit{H}$, and the gaps of adjacent spiked eigenvalues have a constant lower bound.
\end{remark}
The rest of the paper is arranged as follows: In Section \ref{mainresults}, we give the separation condition and state our main results of the asymptotic law and CLT for the extreme eigenvectors of the generalized spiked model in a general case. In Section \ref{appl}, we discuss the application of our results. In the Section \ref{s}, simulations are conducted to evaluate our work. Then in Section \ref{proof}, after some explanations of the truncation procedure and useful lemmas, we prove our main result, Theorem \ref{th2}, based on a key technical PG4MT for some statistics and RMT. The proof of Theorem \ref{th1} and \ref{th3} are also stated in Section \ref{proof}.  Finally, we conclude in the Section \ref{conclusion}.

Throughout the paper, $\text{tr}\bbA$ denotes the trace of a square matrix $\bbA$. For vectors $\bbv$ and $\bbw$, we use $\langle\bbv,\bbw\rangle=\bbv^*\bbw$ to denote their scalar product. It is important to note that the vectors and matrices considered in this paper are real, so $\bbv^*\bbw=\bbv^\top\bbw$ and $\bbA^*=\bbA^\top$. Additionally, we use $\left\|\bbv\right\| $ to represent the $\ell^2$ norm for a vector $\bbv$. Moreover, we use $X_n\xrightarrow{a.s.}X$, $X_n\xrightarrow{P}X$ and $X_n\xrightarrow{D}X$ to represent the convergence of a sequence \{$X_n$\} of random variables to the random variable $X$ almost surely, in probability and in distribution, respectively.

\section{Main results}\label{mainresults}

In this section, we give our main results. The sample eigenvalues of the generalized spiked sample covariance matrix $\bbB$ are sorted in descending order as$$l_1(\bbB),\cdots,l_j(\bbB),\cdots,l_p(\bbB).$$
Since the spiked eigenvalues may tend to infinity in our work, from Proposition 2.1 in \cite{2021jiangbai} and Theorem 1 in \cite{PG4MT}, we know that for all $j\in\mathcal{J}_k$, $\{l_j(\bbB)/\psi_k - 1\}$ almost surely converges to 0 and propose a CLT for \\$\left( \sqrt{n}\left( \dfrac{l_j(\bbB)}{\psi_{nk}}-1\right), j\in\mathcal{J}_k \right) ^{\prime}$, respectively, where 	
\begin{equation}\label{psink}
	\psi_{nk}:=\psi_n(d_k^2)=d_k^2\left( 1+c_n\int\dfrac{t}{d_k^2-t}\text{d}H_n(t)\right).
\end{equation} Since the convergence of $c_n \to c$ and $H_n \to H$ may be very slow, the difference $\sqrt{n}(l_j-\psi_k)$ may not have a limiting distribution. So, we usually use $\psi_{nk}$ instead of $\psi_k$, in particular during the process of deriving the CLT.

In the sequel, we fix an $i$ and consider a (possibly) multiple $d_k^2, i\in\mathcal{J}_k$. In order to study the generalized components of the eigenvector of the $|\mathcal{J}_k|$-fold multiple $d_k^2$, we introduce
$$\left\| \pmb\xi_{i1k}\right\| ^2=\sum_{j\in\mathcal{J}_k}\langle \bbv_j,\bbf_i\rangle^2,$$ where $\bbv_i$ be the eigenvector associated with the $i$-th eigenvalue of $\pmb\Sigma$.
For example, if $d_1^2$ is a simple eigenvalue, that $\mathcal{J}_1=\{1\}$, then $\left\| \pmb\xi_{111}\right\| ^2=\left\langle \bbv_1,\bbf_1\right\rangle^2.$

\begin{remark}
	Since the perturbations of the sample spiked eigenvalues fundamentally overlap, the correlation inference based on the distinct eigenvalues is considered invalid. Moreover, the eigenvector is very sensitive to the gap of the eigenvalue, so the inference of $\bbv_i$ based on $\bbf_i$ is also unreliable. Hence, the separation condition can be considered as the minimum requirement for spike detection. The inequality (\ref{sc}) is the so-called separation condition which guarantees that the distinct (possibly multiple) $d_i^2$ are well separated to prevent the sample spikes $l_i$ corresponding to distinct $d_i^2$ tending to a common limit.
\end{remark}
With the above necessary notations and assumptions, we now state the main theorem about the CLT of the eigenvectors corresponding to the spiked eigenvalues of the generalized spiked sample covariance matrix $\bbB$.

\begin{theorem}\label{th2}
	Suppose that Assumptions \ref{tail}-\ref{distant} hold. Fix an $i\in\mathcal{J}_k$ and let $\bbv_j$ be the $j$th column of matrix $\bbV$. Then,
	\begin{equation}
		\sqrt{n}\left( \left\| \pmb\xi_{i1k}\right\| ^2-\dfrac{1}{1+d_k^2\left( \underline{m}_{n1}(\psi_{nk})+\psi_k \underline{m}_{n2}(\psi_{nk})\right)}\right) 
	\end{equation}
	converges weakly to a Gaussian variable whose mean is zero and variance is $$\dfrac{\theta_k}{\left( 1+d_k^2\left( \underline{m}_1(\psi_k)+\psi_k \underline{m}_2(\psi_k)\right)\right) ^4},$$
	where $\psi_k$ and $\psi_{nk}$ are defined  in (\ref{psi}) and (\ref{psink}) respectively, $\theta_k=d_k^4\left( \dfrac{2}{m_k}\sigma_k+\dfrac{\kappa_{x,k}}{m_k^2}\rho_k^2\right) $, $\sigma_k=\underline{m}_2(\psi_k)+2\psi_k\underline{m}_3(\psi_k)+\psi_k^2\underline{m}_4(\psi_k)$, $\rho_k=\underline{m}_1(\psi_{k})+\psi_{k}\underline{m}_{2}(\psi_{k})$, and the definition of $\kappa_{x,k}$ is given in (\ref{ux}). Moreover,
	$$\underline{m}_j(z)=\int\dfrac{1}{(x-z)^j}\text{d}\underline{F}(x),$$
	$j=1,2,3,4$, and $\underline{F}$ is the LSD of $\bbX^*\bbU_2\bbD_2^2\bbU_2\bbX/n$.
\end{theorem}

\begin{remark}
	Because the convergence of $\left\| \pmb{\xi}_{i1k}\right\| ^2$ to $1/( 1+d_k^2( \underline{m}_1(\psi_k)+\psi_k \underline{m}_2(\psi_k))) $ can be arbitrarily slow, we have to change the center to  $1/( 1+d_k^2( \underline{m}_{n1}(\psi_{nk})+\psi_{nk} \underline{m}_{n2}(\psi_{nk}))) $, where $\psi_{nk}=\psi_n(d_k^2)$ and $\underline{m}_{n1}(z)=\underline{m}_n(z)$ which is defined similarly as $m_n$. And $\underline{m}_{n2}(z)=\underline{m}\textquotesingle_{n1}(z)$.
	Moreover, it is noteworthy that if the 4th moment does not exist, we only need all $\kappa_{x,k}$ are finite.  when $\pmb{\Sigma}=\bbI+\sum_{i=1}^{r}(d_i^2-1)\bbv_i\bbv_i^*$ with $d_1^2 >\dots>d_r^2$, the model reduces to the one investigated in \cite{2020baoding}. In this study, the result is the joint distribution of the extreme eigenvalues and the generalized components of the associated eigenvectors.
\end{remark}

As a corollary of Theorem \ref{th2}, our second result is the convergence for the spiked eigenvector for a general case.

\begin{corollary} \label{th1}
	Under the same assumptions as Theorem \ref{th2}, then the following convergence hold:
	\begin{equation}\label{limit}
		\left\| \pmb\xi_{i1k}\right\| ^2 \xrightarrow{P} \dfrac{1}{1+d_k^2(\underline{m}_1(\psi_k)+\psi_k\underline{m}_2(\psi_k))}.
	\end{equation}
\end{corollary}

\begin{remark}
	It is worth mentioning that we can also consider the case where the spiked eigenvalues of $\pmb{\Sigma}$ are unbounded. If $d_k^2 \to \infty$, since $\psi(d_k^2)=d_k^2( 1+c\int\dfrac{t}{d_k^2-t}\text{d}H(t))$, then $\dfrac{d_k^2}{\psi(d_k^2)}\xrightarrow{P} 1$. By (\ref{limit}), we have $\left\| \pmb\xi_{i1k}\right\| ^2 \xrightarrow{P} 1$. The same conclusion can also be obtained from (\ref{u2}). 
\end{remark}

\begin{remark}
	In this paper, we consider the case where the 4th moment may unnecessarily exist, as indicated by Assumption \ref{tail}. Moreover, it should hold that $\left\| \pmb\xi_{i1k}\right\| ^2 \xrightarrow{a.s.} \dfrac{1}{1+d_k^2(\underline{m}_1(\psi_k)+\psi_k\underline{m}_2(\psi_k))}$ under the bounded 4th moment assumption. However, we do not pursue this under Assumption \ref{tail} because our main focus is on the CLT of extreme eigenvectors.  
\end{remark}

\section{Statistical inference}\label{appl}

In this section, we apply our results to a statistical application. We focus our discussion on the hypothesis testing regarding the eigenspaces of covariance matrices. The eigenspaces of covariance matrices hold significant importance in various statistical methodologies and computational algorithms, particularly in the fields of data analysis and machine learning. 

In this section, we consider a generic index set $\mathcal{J}\subset\{1, \dots, M\}$, which shall be regarded as an extension of $\mathcal{J}_k$. Further, we set 
\begin{equation}\label{Z0}
	Z_0=\sum_{i\in \mathcal{J}}{\bbv_i\bbv_i^*}.
\end{equation}

Specifically, researchers are particularly interested in testing the following hypothesis:
\begin{equation}\label{test}
	H_0 : Z=Z_0 \ \ \ \ \text{vs}\ \ \ \  H_1:Z\neq Z_0
\end{equation}
for a given projection $Z_0$ defined in (\ref{Z0}). Next, we propose accurate and powerful statistics for the testing problem (\ref{test}). For (\ref{test}), we construct test statistics using the plug-in estimators:
\begin{equation}\label{T}
	\mathcal{T}:=\sum_{i\in \mathcal{J}}\left( \langle\bbv_i,\sum_{i\in \mathcal{J}}\bbf_i\bbf_i^*\bbv_i\rangle-\nu(\widehat{d_i^2})\right) ,
\end{equation}	
where $$\nu(d)=\dfrac{1}{1+d^2(\underline{m}_1(\psi)+\psi_k\underline{m}_2(\psi))}$$
and $\widehat{d_k^2}=\gamma(l_i), i\in\mathcal{J}$  is a nonlinear shrinkers of the sample eigenvalues in (\ref{estimate}). 

In fact, the LSD of $\dfrac{1}{n}\bbX^*\bbU_2\bbD_2^2\bbU_2^*\bbX$ is approximately the same as the one of $\dfrac{1}{n}\bbX^*\bbU\bbD^2\bbU^*\bbX$ because of the number of spikes is fixed. Then, we obtain a good estimator of $m(l_i)$ as below
$$\hat{m}(l_i)=\dfrac{1}{p-\mathcal{J}_0}\sum_{j\notin\mathcal{J}_0;j=1}^p \dfrac{1}{l_j-l_i},$$
where we define $r_{ij}=|l_i-l_j|/\max(l_i,l_j)$ and $\mathcal{J}_0=\{i:r_{ij}\leq0.2 \ \text{for any} \ j=1,\dots,p\}$, $\tilde{c}=(p-\mathcal{J}_0)/n$. The chosen setting $\mathcal{J}_0$ aims to mitigate the impact of multiple roots, as their presence can lead to inaccurate estimations of population spikes. By carefully selecting the setting, we can minimize the interference caused by multiple roots and improve the accuracy of our estimations for population spikes. The constant 0.2 serves as a suitable threshold value of the ratio, as indicated by the simulated results in \cite{2021jiangbai}. Moreover, we adopt $$\hat{\underline{m}}(l_i)=-\dfrac{1-\tilde{c}}{l_i}+\tilde{c}\hat{m}(l_i).$$ Then, we get the estimator of $d_k^2$, which can be represented as 
\begin{equation}\label{estimate}
	\widehat{d_k^2}=\gamma(l_i)=-\dfrac{1}{\hat{\underline{m}}(l_i)}, i\in\mathcal{J}_k.
\end{equation}

Next, for simplicity, we will mainly work with a single spike model ($M=1$). Then, for test (\ref{test}) and by (\ref{T}), we study the test statistics
\begin{equation}\label{T1}
	\mathcal{T}_1:=\dfrac{1}{d_1^2}\left( \langle\bbv_1,\bbf_1\rangle^2-\nu(\widehat{d_1^2})\right) .
\end{equation}
We remark here that the spike $d_1^2$ can be unbounded.

We record the results regarding the asymptotic distribution of (\ref{T1}) in the following theorem and postpone its proof to section \ref{proof}.

\begin{theorem}\label{th3}
	Under the existence of the 4th moment, then
	$$\dfrac{\sqrt{n}\mathcal{T}_1}{\sqrt{\vartheta(d_1^2)}} \xrightarrow{D}N(0,1),$$
	where $\vartheta(d_1^2)=\tau(d_1^2)\tilde{G}(\psi_1)$, the definitions of the $\tilde{G}(\psi_1)$ and the function $\tau(d_1^2)$ are given in Section \ref{proof}.
\end{theorem}

By Theorem \ref{th3} and Theorem 4.1 in \cite{2012baiyao}, we can construct a pivotal statistic. Denote
\begin{equation}\label{TT1}
	\mathcal{T}_2=\dfrac{\sqrt{n}\mathcal{T}_1}{\sqrt{\vartheta(\widehat{d_1^2})}},\ \widehat{d_1^2}=-\dfrac{1}{\hat{\underline{m}}(l_1)}.
\end{equation}
We mention that (\ref{TT1}) is adaptive to the $d_1^2$ by using its estimators (\ref{estimate}). We summarize the distribution of $\mathcal{T}_2$ in the corollary below, whose proof will also be postponed to Section \ref{proof}.

\begin{corollary}\label{corollary}
	Under the assumptions of Theorem \ref{th3}, we have that
	$$\mathcal{T}_2\xrightarrow{D} N(0,1).$$
\end{corollary}

\section{Simulation}\label{s}
This section conducts simulations to demonstrate the performance of our result. For simplicity, we only give the results of two cases of the population covariance matrix. We generate samples based on sample size $n=1000$ when $p=100$. They are detailed below:

\textit{Case \uppercase\expandafter{\romannumeral1}}: Let $\bbT$ = diag$(4, 3, 3, 0.2, 0.2, 0.1, 1, \dots, 1)$ is a finite-rank perturbation of matrix $\bbI_p$ with the spikes (4, 3, 0.2, 0.1) of the multiplicity (1, 2, 2, 1), where the Assumption \ref{aU} is satisfied.

\textit{Case \uppercase\expandafter{\romannumeral2}}: Let $\bbT=\pmb{\Gamma\Lambda\Gamma^*}$ is a general positive definite matrix, where $\pmb{\Lambda}$ = diag $(4, 3, 0.2, 0.1, 0, \dots, 0)$ and $\pmb{\Gamma}$ is the following matrix

$$\begin{pmatrix}
	1 & \tau & 0 & \cdots & 0\\
	\tau & 1 & \tau & \cdots & 0\\
	\vdots & \vdots & \vdots & \ddots & \vdots\\
	0 & 0 & \cdots & \tau & 1
\end{pmatrix}
,$$
where $\tau=0.5$.

For each case, \{$x_{ij}$\} are i.i.d. samples from standard Gaussian population. As described above, we have the spikes 4, 3, 0.2 and 0.1.

For the single population spikes $d_1^2=4$ of \textit{Case \uppercase\expandafter{\romannumeral1}} (see Figure \ref{Figggg1}), we get the limiting distributions
\begin{equation*}
	\sqrt{n}\left( \left\| \pmb\xi_{111}\right\| ^2-\mu_1 \right) \xrightarrow{D}N\left( 0,\sigma_1^2\right) ,
\end{equation*}
where $\mu_1=0.9496$, $\sigma_1=2.1786$.

For the spikes $d_2^2=3$ with multiplicity 2, we obtain that the two distributions

\begin{equation*}
	\begin{aligned}
		&\sqrt{n}\left( \left\| \pmb\xi_{212}\right\| ^2+\left\| \pmb\xi_{312}\right\| ^2-\mu_2 \right) \xrightarrow{D}N\left( 0,\sigma_2^2\right) ,\\
	\end{aligned}
\end{equation*}
where $\mu_2=0.9107$, $\sigma_2=1.7409$.

\begin{figure}
	\centering
	\subfigure[$d_1^2=4$]{\includegraphics[height=3.5cm, width=4cm]{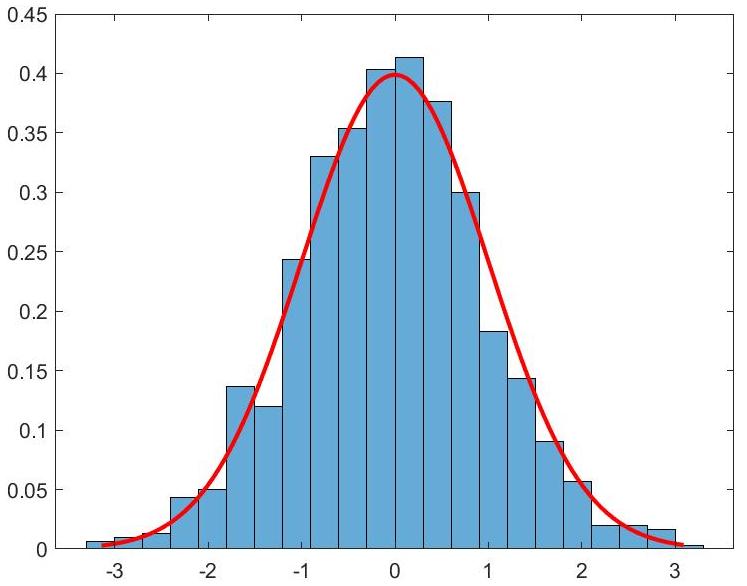}}
	\subfigure[$d_2^2=3$ with multiplicity 2]{\includegraphics[height=3.5cm, width=4cm]{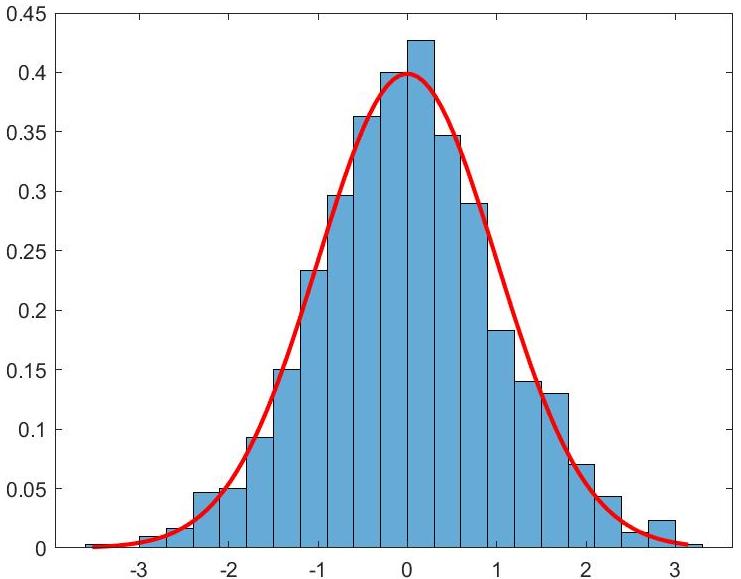}}
	\caption{For the \textit{Case \uppercase\expandafter{\romannumeral1}}, simulated empirical distributions for $\langle \bbv_1,\bbf_1\rangle^2$ under $d_1^2=4$ and $\langle\bbv_2,\bbf_2\rangle^2+\langle\bbv_3,\bbf_2\rangle^2$ under $d_2^2=3$. Here, we report our results based on 1000 simulations with Gaussian random variables.}\label{Figggg1}
\end{figure}
Similarly, for the \textit{Case \uppercase\expandafter{\romannumeral2}} (see Figure \ref{Figggg2}), using the Theorem \ref{th2} can calculate the limiting distributions.

\begin{figure}
	\centering
	\subfigure[$d_1^2=4$]{\includegraphics[height=3.5cm, width=4cm]{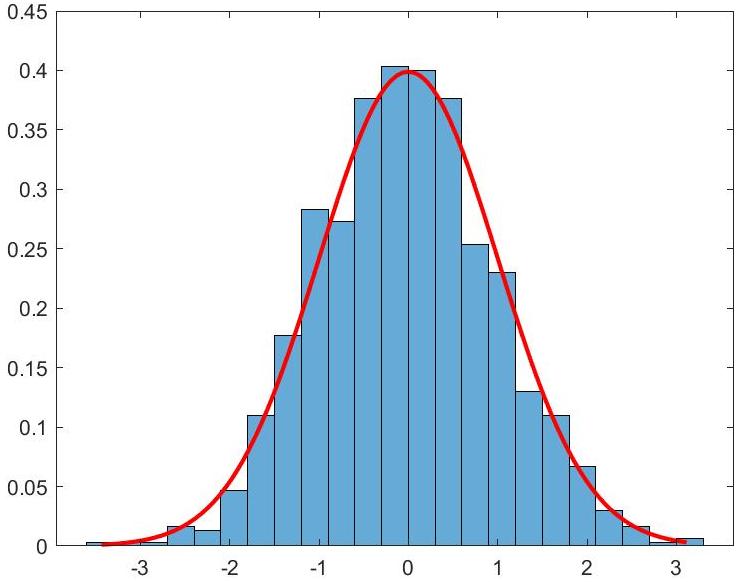}}
	\subfigure[$d_4^2=0.1$]{\includegraphics[height=3.5cm, width=4cm]{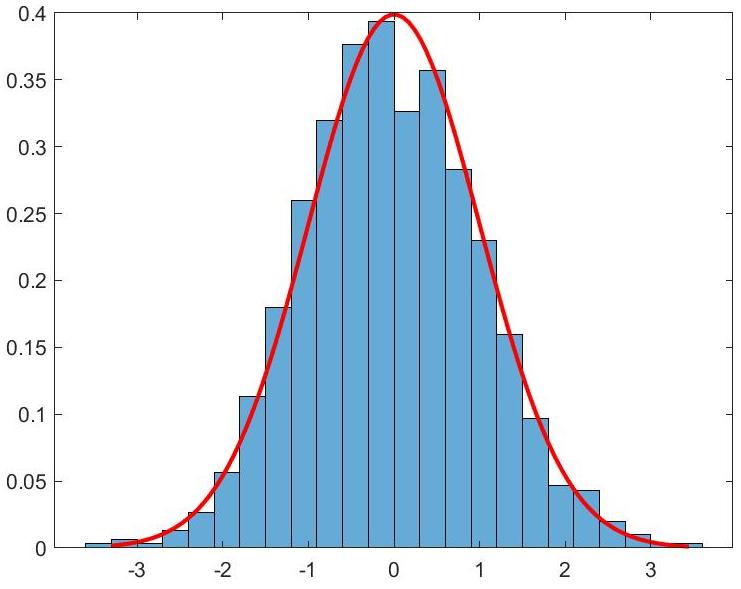}}
	\caption{For the \textit{Case \uppercase\expandafter{\romannumeral2}}, simulated empirical distributions for $\langle \bbv_1,\bbf_1\rangle^2$ and $\langle \bbv_4,\bbf_4\rangle^2$ under $d_1^2=4$ and $d_4^2=0.1$, respectively. Here, we report our results based on 1000 simulations with Gaussian random variables.}\label{Figggg2}
\end{figure}
As shown in the calculations and simulations, our approach extends the limiting distributions for the extreme sample eigenvalues to a generalized spiked population model.

Since $\mathbb{T}_1$ is asymptotically pivotal, we will use (\ref{TT1}) as our statistic to test (\ref{test}).
For the accuracy of the test, we focus on reporting the results with the type I error rate 0.05 under different values of $c=0.1,1,2$ and various choices of the spikes. We consider the hypothesis testing for the eigenspace of $\bbT\bbT^*=\text{diag}(d_1^2,2,\cdots,2,1,\cdots,1)$, where the null is $$Z_0=\bbe_1\bbe_1^*.$$
We will assume the standard Gaussian distribution of entries of $\bbX$. From Table \ref{size}, we observe that our proposed statistic (\ref{TT1}) yields accurate results for different values of $d_1^2$, and this accuracy is robust against different values of $c$.

\begin{table}
	\caption{\centering{Simulated type I error rates under the nominal level 0.05.}}
	\centering
	\begin{tabular}{ccccccc}
		\toprule
		& \multicolumn{3}{c}{$n=200$}&\multicolumn{3}{c}{$n=500$}\\
		\midrule
		&$c=0.1$&$c=1$&$c=2$&$c=0.1$&$c=1$&$c=2$\\
		\midrule
		$d_1^2=5$&0.0530&0.0480&0.0510&0.0530&0.0520&0.0410\\
		$d_1^2=10$&0.0550&0.0460&0.0500&0.0520&0.0510&0.0490\\
		$d_1^2=50$&0.0520&0.0585&0.0570&0.0540&0.0450&0.0520\\
		$d_1^2=100$&0.0530 &0.0545&0.0480&0.0525&0.0505&0.0490\\
		\bottomrule
	\end{tabular}
	\label{size}
\end{table}

Then we investigate the power of the above statistic under the alternative $Z=\bbe_4\bbe_4^*$ for $d_1^2=10$ and $d_1^2=100$ in Table \ref{power}. We report our results under the nominal level 0.05 based on 2000 simulations.

\begin{table}
	\caption{\centering{Power for $d_1^2=10$ and $d_1^2=100$ under the nominal level 0.05.}}
	\centering
	\begin{tabular}{ccccccc}
		\toprule
		& \multicolumn{3}{c}{$d_1^2=10$}&\multicolumn{3}{c}{$d_1^2=100$}\\
		\midrule
		& $n=100$ & $n=200$&$n=500$& $n=100$ & $n=200$&$n=500$\\
		
		\midrule
		$c=0.1$ &0.8610 &0.9060 &0.9390 &0.8780 &0.9120 &0.9390 \\
		$c=1$ &0.8450 &0.9020 &0.9300 &0.8600 &0.8930 &0.9400 \\
		$c=2$ &0.8310 &0.9050 &0.9310 &0.8550 &0.8910 &0.9300 \\
		\bottomrule
	\end{tabular}
	\label{power}
\end{table}

\section{Proofs of the main results}\label{proof}
In this section, we prove the main results in Section \ref{mainresults}.  
We first give some explanations of the truncation procedure as below. We can show that replacing the entries of $\bbX$ with the truncated and renormalized ones, as specified in Assumption \ref{tail}, produces equivalent results.

\subsection{Truncation}
Let $\hat{x}_{ij}=x_{ij}I(|x_{ij}|<\eta_n\sqrt{n})$ and $\tilde{x}_{ij}=(\hat{x}_{ij}-\rE\hat{x}_{ij})/\sigma_n$ with $\sigma_n^2=\rE|\hat{x}_{ij}-\rE\hat{x}_{ij}|^2$, where $\eta_n \to 0$ with a slow rate. 
Consider the truncated samples $\hat{x}_{ij}=x_{ij}I(|x_{ij}|<\eta_n\sqrt{n})$, according to Assumption \ref{tail}, we have
$$\begin{aligned}
	\text{P}\left( \bbB\neq\hat{\bbB}\right) &\leqslant \sum_{i,j}\text{P}\left( |x_{ij}|\geq\eta_n\sqrt{n}\right)\\
	&=np\text{P}\left( |x_{11}|\geq\eta_n\sqrt{n}\right) \to 0,\ \ \ \ as\ n,p\to\infty.
\end{aligned} $$

Next, define the truncated and renormalized sample covariance matrix as 
$$\tilde{\bbB}=\frac{1}{n}\bbT\tilde{\bbX}\tilde{\bbX}^*\bbT^*.$$
Then, by using the same approach and bounds that are used in the proof of Lemma 2.7 of Bai \cite{1999bai} and elementary calculation, we have, for each $j=1,2,\dots,n$,
$$\begin{aligned}
	\bbv_i^*\hat{\bbB}\bbv_i-\bbv_i^*\tilde{\bbB}\bbv_i&=\bbv_i^*\left( \sum_{j=1}^{n}\hat{l}_j\hat{\bbf}_j\hat{\bbf}_j^*\right) \bbv_i-\bbv_i^*\left( \sum_{j=1}^{n}\tilde{l}_j\tilde{\bbf}_j\tilde{\bbf}_j^*\right)\bbv_i\\
	&=\sum_{j=1}^{n}\hat{l}_j\langle \bbv_i, \hat{\bbf}_j\rangle ^2-\sum_{j=1}^{n}\tilde{l}_j\langle \bbv_i, \tilde{\bbf}_j\rangle ^2\\
	&=\sum_{j=1}^{n}\left( \hat{l}_j-\tilde{l}_j\right) \langle \bbv_i, \hat{\bbf}_j\rangle ^2+\sum_{j=1}^{n}\tilde{l}_j\left(\langle \bbv_i, \hat{\bbf}_j\rangle ^2-\langle \bbv_i, \tilde{\bbf}_j\rangle ^2 \right)\\
	&\leq  \left\| \tilde{\bbB}-\hat{\bbB}\right\|.
\end{aligned}$$

From Assumption \ref{tail} we have,
$$\begin{aligned}
	\left| \sigma_n^2-1\right| &=\left| E|\hat{x}_{ij}-E\hat{x}_{ij}|^2-E|x_{ij}|^2\right|\\
	&\leq 2E|x_{ij}|^2 I\left(|x_{ij}|\geq \eta_n\sqrt{n} \right)\\
	&=2\int_{\eta_n \sqrt{n}}^{\infty}x^2\text{d}\text{P}(|x_{11}|<x)= -2\int_{\eta_n \sqrt{n}}^{\infty}x^2\text{d}\text{P}(|x_{11}|\geq x)\\
	&=2\int_{\eta_n \sqrt{n}}^{\infty}2x\text{P}(|x_{11}|\geq x)\text{d}x=2\int_{\eta_n \sqrt{n}}^{\infty}\dfrac{2x^4\text{P}(|x_{11}|\geq x)}{x^3}\text{d}x\\
	&=o(1)\int_{\eta_n \sqrt{n}}^{\infty}\dfrac{1}{x^3}\text{d}x=o(n^{-1}).
\end{aligned}$$

Moreover, \begin{equation*}\label{Ex}
	\left| \text{E}\hat{x}_{11}\right|= \left| \int_{\eta_n \sqrt{n}}^{\infty} x \text{d}\text{P}\left( \left| x_{11}\right|\geq x \right)\right| =o\left( n^{-3/2}\right) ,
\end{equation*}
and

\begin{equation}\label{4th}
	\begin{aligned}
		| \text{E}\tilde{x}_{ij}^4|&=\text{E}\left\lbrace x_{ij}^4 I\left(|x_{ij}|< \eta_n\sqrt{n} \right)  \right\rbrace\\ &=-\int_{0}^{\eta_n \sqrt{n}} x^4\text{d}\text{P}\left( \left| x_{11}\right|\geq x \right)=\int_{0}^{\eta_n \sqrt{n}} \dfrac{4x^4\text{P}\left( \left| x_{11}\right|\geq x \right)}{x}\text{d}x\\
		&=\int_{0}^{1} 4x^3\text{P}\left( \left| x_{11}\right|\geq x \right)\text{d}x + \int_{1}^{\eta_n \sqrt{n}} \dfrac{4x^4\text{P}\left( \left| x_{11}\right|\geq x \right)}{x}\text{d}x\\
		&=O(1)+\int_{1}^{\eta_n \sqrt{n}}\dfrac{o(1)}{x}\text{d}x=O(\log n).
	\end{aligned}
\end{equation}	
These give us
$$\begin{aligned}
	\left\| \tilde{\bbB}-\hat{\bbB}\right\| &\leq \left\| \tilde{\bbX}-\hat{\bbX}\right\| \left(\left\|\tilde{\bbX}\right\| +\left\| \hat{\bbX}\right\|   \right) \\
	&\leq \left( \left\| \text{E}\hat{\bbX}\right\| +\left\| (1-\sigma_n)\tilde{\bbX}\right\| \right) \left( \left\|\tilde{\bbX}\right\| +\left\| \hat{\bbX}\right\|   \right)\\
	&\leq \dfrac{1}{\sqrt{n}}\left| \text{E}\hat{x}_{11}\right| \cdot \min(p,n)+\dfrac{1-\sigma_n^2}{1+\sigma_n}\left\| \tilde{\bbB}\right\| \\
	&=o_{a.s.}\left( n^{-1}\right) .
\end{aligned}$$

From the above estimates, we obtain
$$\langle \bbv_j, \bbf_j\rangle ^2 = \langle \bbv_j, \tilde{\bbf}_j\rangle ^2+o_p(1).$$

Therefore, we only need to consider the limiting distribution of the eigenvectors corresponding to the spiked eigenvalues of $\tilde{\bbB}$, which consists of the entries truncated at $\eta_n\sqrt{n}$, centralized and renormalized. For simplicity, in real cases, $|x_{ij}|<\eta_n\sqrt{n}$, $\rE x_{ij}=0$, and $\rE |x_{ij}^2|=1$ is equivalent to Assumption \ref{tail}. 

\subsection{Proof of Theorem \ref{th2}}
First, we will introduce the following Lemma \ref{cor}, which plays an important role in the proof. 

\begin{lemma}\label{cor}
	Suppose that Assumptions \ref{tail}-\ref{distant} hold. In addition, $\bbU_{1k}^*\bbX=\bbY_k\in\mathbb{R}^{m_k\times n}$ and the entries of $\bbU_2^*\bbX=\bbZ\in\mathbb{R}^{(p-M)\times n}$ can be assumed as i.i.d. $N(0,1)$. Then,
	\begin{equation}\label{trclt}
		\resizebox{0.9\hsize}{!}{$
		\begin{aligned}
			&\dfrac{d_k^2}{n^{3/2}} \left( \dfrac{1}{m_k}{\rm tr}\left( \bbY_k\bbZ^*\bbD_2\left( \frac{1}{n}\bbD_2\bbZ\bbZ^*\bbD_2-l_i\bbI_{p-M}\right) ^{-2}\bbD_2\bbZ\bbY_k^*\right)-{\rm tr}\left( \bbZ^*\bbD_2\left( \frac{1}{n}\bbD_2\bbZ\bbZ^*\bbD_2-l_i\bbI_{p-M}\right) ^{-2}\bbD_2\bbZ\right)\right)  
		\end{aligned}$}
	\end{equation}
	tends to normal distribution with mean 0 and variance $\theta_k=d_k^4\left( \dfrac{2}{m_k}\sigma_k+\dfrac{\kappa_{x}}{m_k^2}\rho_k^2\right) $, where $\kappa_{x,k}$ is given in (\ref{ux}), and the definition of $\sigma_k$ and $\rho_k$ are given in Theorem \ref{th2}.
\end{lemma}

To establish Theorem \ref{th2}, we require Lemma \ref{PG4MT}, the proof of which is presented in \cite{PG4MT}.
\begin{lemma}[PG4MT, Theorem 1 in \cite{PG4MT}]\label{PG4MT}
	Assuming that $\bbX$ and $\bbY$ are two sets of double arrays satisfying Assumptions \ref{tail}-\ref{distant}, then it holds that $\Omega_M(\psi_{n,k},\bbX,\bbX)$ and  $\Omega_M(\psi_{n,k},\bbX,\bbY)$ have the same limiting distribution, provided one of them has.
\end{lemma}

\begin{remark}
	Lemma \ref{PG4MT} illustrates the basic idea and procedures of the PG4MT to show the universality of a limiting result related to the large dimensional random matrices. By the PG4MT, we can show the universality of the asymptotic law for spiked eigenvalues of a generalized spiked covariance matrix.
\end{remark}

Now, we start to prove Lemma \ref{cor}.
\begin{proof}[Proof of Lemma \ref{cor}]

	Consider $\dfrac{1}{\sqrt{n} \cdot m_k}{\rm tr}(\bbY_k\bbA\bbY_k^* )$, note that the mean of the $\dfrac{1}{\sqrt{n} \cdot m_k}{\rm tr}\left(\bbY_k\bbA\bbY_k^* \right)$ is
	\begin{equation}\label{E}
		E\left( \dfrac{1}{\sqrt{n} \cdot m_k}{\rm tr}\left(\bbY_k\bbA\bbY_k^* \right)\right) =\dfrac{1}{\sqrt{n} \cdot m_k}{\rm tr}\left( \bbA\cdot E\bbY_k^*\bbY_k\right) =\dfrac{1}{\sqrt{n}}{\rm tr}\left( \bbA\right),
	\end{equation}
	where $\bbA=\dfrac{1}{n}\bbZ^*\bbD_2\left( \dfrac{1}{n}\bbD_2\bbZ\bbZ^*\bbD_2-l_i\bbI_{p-M}\right) ^{-2}\bbD_2\bbZ$. 
	
	Next, we show that $a_{jj}$ are asymptotically identical with $\underline{m}_1(\psi_{k})+\psi_{k}\underline{m}_{2}(\psi_{k})$.
	In fact, we have
	\begin{equation}\label{ajj}
		\begin{aligned}
		a_{jj}&=\left[  \dfrac{1}{n}\bbZ^*\bbD_2\left( \dfrac{1}{n}\bbD_2\bbZ\bbZ^*\bbD_2-l_i\bbI_{p-M}\right) ^{-2}\bbD_2\bbZ\right]  _{jj}\\
		&=\left[  \left( \dfrac{1}{n}\bbZ^*\bbD_2^2\bbZ-l_i\bbI_{n}\right) ^{-1}\right]  _{jj}+l_i\left[  \left( \dfrac{1}{n}\bbZ^*\bbD_2^2\bbZ-l_i\bbI_{n}\right) ^{-2}\right]  _{jj}\\
		&=\left[  \left( \dfrac{1}{n}\bbZ^*\bbD_2^2\bbZ-l_i\bbI_{n}\right) ^{-1}\right]  _{jj}+l_i\left[\left(   \left( \dfrac{1}{n}\bbZ^*\bbD_2^2\bbZ-l_i\bbI_{n}\right) ^{-1}\right) '\right]  _{jj}\\
		&\xrightarrow{P}\underline{m}_1(\psi_{k})+\psi_{k}\underline{m}_{2}(\psi_{k}),
		\end{aligned}
	\end{equation}
where the first step uses the identity $\bbC\left( \bbC^*\bbC-\lambda\bbI\right) ^{-2}\bbC^*=\left(\bbC\bbC^*-\lambda\bbI \right) ^{-1}+\lambda\left( \bbC\bbC^*-\lambda\bbI\right)^{-2} $ and the third step follows from (4.7) in \cite{2020baoding}.
	
	Now, conditional on $\bbZ$, then we have
	\begin{equation}\label{var}
		\begin{aligned}
			&E\left( \dfrac{1}{\sqrt{n} \cdot m_k}{\rm tr}(\bbY_k\bbA\bbY_k^* )-\dfrac{1}{\sqrt{n}}{\rm tr}\left( \bbA\right)\right)^2\\
			=&\dfrac{1}{nm_k^2}\left( \sum_{i\in\mathcal{J}_k}\sum_{j=1}^n\sum_{k=1}^n y_{ij}a_{jk}y_{ik}-\sum_{i\in\mathcal{J}_k}\sum_{q=1}^n a_{qq}\right) ^2\\
			=&\dfrac{1}{nm_k^2}\left( \sum_{i\in\mathcal{J}_k}\sum_{j=1}^n\left( \sum_{t,l}^p u_{it}u_{il}x_{jt}x_{jl}-1\right) a_{jj}\right) ^2+\dfrac{1}{nm_k^2}\left( \sum_{i\in\mathcal{J}_k}\sum_{j\neq k}^n\sum_{t,l}^p u_{it}u_{il}x_{jt}x_{kl}a_{jk}\right) ^2\\
			=&\dfrac{1}{nm_k^2}\sum_{i_1,i_2}\sum_{j_1,j_2} a_{j_1j_1}a_{j_2j_2}\left( \sum_{t_1,l_1}u_{i_1t_1}u_{i_1l_1}x_{j_1t_1}x_{j_1l_1}-1\right) \left( \sum_{t_2,l_2}u_{i_2t_2}u_{i_2l_2}x_{j_2t_2}x_{j_2l_2}-1\right)\\
			&~~+\dfrac{1}{nm_k^2}\sum_{i_1,i_2}\sum_{j_1\neq k_1}\sum_{j_2\neq k_2}\sum_{t_1,l_1}\sum_{t_2,l_2} a_{j_1k_1}a_{j_2k_2}u_{i_1t_1}u_{i_2t_2}u_{i_1l_1}u_{i_2l_2}x_{j_1t_1}x_{j_2t_2}x_{k_1l_1}x_{k_2l_2}\\
			=&\dfrac{1}{m_k^2}\left( \sum_{i\in\mathcal{J}_k}\sum_{t=1}^p u_{it}^4+\sum_{\stackrel{i_1\neq i_2}{i_1,i_2\in\mathcal{J}_k}} \sum_{t=1}^p u_{i_1t}^2u_{i_2t}^2\right) \left( Ex_{11}^4\mathit{I}(|x_{11}|<\sqrt{n})-3\right)\cdot\dfrac{1}{n} \sum_{j=1}^n a_{jj}^2+\dfrac{2}{nm_k}\text{tr}\bbA^2,
		\end{aligned}
	\end{equation}
	with
	
		\begin{align}\label{trA2}
			&\ \ \ \ \dfrac{1}{n}{\rm tr}\bbA^2\nonumber\\&=\dfrac{1}{n}{\rm tr}\left( \dfrac{1}{n}\bbZ^*\bbD_2\left( \dfrac{1}{n}\bbD_2\bbZ\bbZ^*\bbD_2-l_i\bbI_{p-M}\right) ^{-2}\bbD_2\bbZ\right) ^2\nonumber\\
			&=\dfrac{1}{n}{\rm tr}\left( \bbZ^*\bbD_2\left( \dfrac{1}{n}\bbD_2\bbZ\bbZ^*\bbD_2-l_i\bbI_{p-M}\right) ^{-2}\right.\nonumber\\
			&~~~~~~~~~~~~\left.\cdot\left( \dfrac{1}{n}\bbD_2\bbZ\bbZ^*\bbD_2-l_i\bbI_{p-M}+l_i\bbI_{p-M}\right)\left( \dfrac{1}{n}\bbD_2\bbZ\bbZ^*\bbD_2-l_i\bbI_{p-M}\right) ^{-2}\dfrac{1}{n}\bbZ^*\bbD_2\right) \\
			&=\dfrac{1}{n}{\rm tr}\left( \left( \dfrac{1}{n}\bbD_2\bbZ\bbZ^*\bbD_2-l_i\bbI_{p-M}\right) ^{-3}\dfrac{1}{n}\bbD_2\bbZ\bbZ^*\bbD_2\right)
			+\dfrac{1}{n}l_i\cdot{\rm tr}\left( \left( \dfrac{1}{n}\bbD_2\bbZ\bbZ^*\bbD_2-l_i\bbI_{p-M}\right) ^{-4}\dfrac{1}{n}\bbD_2\bbZ\bbZ^*\bbD_2\right)\nonumber\\
			&=\dfrac{1}{n}{\rm tr}\left( \dfrac{1}{n}\bbD_2\bbZ\bbZ^*\bbD_2-l_i\bbI_{p-M}\right) ^{-2}+\dfrac{2l_i}{n}{\rm tr}\left( \dfrac{1}{n}\bbD_2\bbZ\bbZ^*\bbD_2-l_i\bbI_{p-M}\right) ^{-3}
			+\dfrac{l_i^2}{n}{\rm tr}\left( \dfrac{1}{n}\bbD_2\bbZ\bbZ^*\bbD_2-l_i\bbI_{p-M}\right) ^{-4}\nonumber\\
			&\xrightarrow{P}\underline{m}_2(\psi_k)+2\psi_k\underline{m}_3(\psi_k)+\psi_k^2\underline{m}_4(\psi_k).\nonumber
		\end{align}

Define \begin{equation}\label{sr}
	\sigma_k=\underline{m}_2(\psi_k)+2\psi_k\underline{m}_3(\psi_k)+\psi_k^2\underline{m}_4(\psi_k)\ \ \ \ \text{and} \ \ \ \ \rho_k=\underline{m}_1(\psi_{k})+\psi_{k}\underline{m}_{2}(\psi_{k}).
\end{equation}
Then, by (\ref{ajj}), (\ref{var}) and (\ref{trA2}), we obtain
\begin{equation}\label{var'}
	Var\left( \dfrac{1}{\sqrt{n} \cdot m_k}{\rm tr}\left(\bbY_k\bbA\bbY_k^* \right)\right)\xrightarrow{P}\dfrac{2}{m_k}\sigma_k+\dfrac{\kappa_{x,k}}{m_k^2}\rho_k^2,
\end{equation}
where $\kappa_{x,k}=\lim\left( \sum_{i\in\mathcal{J}_k}\sum_{t=1}^p u_{it}^4+\sum_{\stackrel{i_1\neq i_2}{i_1,i_2\in\mathcal{J}_k}} \sum_{t=1}^p u_{i_1t}^2u_{i_2t}^2\right) \left( Ex_{11}^4\mathit{I}(|x_{11}|<\sqrt{n})-3\right)$.
By applying the central limit theorem of Theorem 7.1 in \cite{2008baiyao}, we can obtain $\dfrac{1}{\sqrt{n} \cdot m_k}{\rm tr}\left(\bbY_k\bbA\bbY_k^* \right)$ is a linear combination  of $\left( \dfrac{1}{\sqrt{n}}\bbY_k\bbA\bbY_k^*\right) _{ii}$, whose joint distribution tends to normal distribution. Therefore, combining (\ref{E}) and (\ref{var'}), we complete the proof.
\end{proof}
\begin{proof}[Proof of Theorem \ref{th2}]
	Let $l_i$ be an eigenvalue of $\bbB$ with corresponding eigenvector $\bbf_i$. Then, we have
	\begin{equation}
		\left( \bbB-l_i\bbI_p\right) \bbf_i=0.
	\end{equation}
	For the generalized spiked covariance matix $\pmb{\Sigma}=\bbT\bbT^*$, consider the corresponding sample covariance matrix $\bbB=\bbT\bbS_n\bbT^*$, where $\bbS_n=\dfrac{1}{n}\bbX\bbX^*$ is the standard sample covariance with sample size $n$. Substituting the singular decomposition of $\bbT$ into $\bbB_n$,we have
	$$\begin{aligned}
		&\left( \bbB-l_i\bbI_p\right) \bbf_i\\=&\left( \bbT\bbS_n\bbT^*-l_i\bbI_p\right)\bbf_i\\
		=&\bbV\left(\bbD\bbU^*\bbS_n\bbU\bbD-l_i\bbI_p\right) \bbV^*\bbf_i\\
		=&\left( \bbV_1,\ \bbV_2\right) \left(\begin{pmatrix} \bbD_1 & 0 \\ 0 & \bbD_2 \end{pmatrix}\begin{pmatrix} \bbU_1^*  \\ \bbU_2^*  \end{pmatrix} \bbS_n \left( \bbU_1,\ \bbU_2\right)\begin{pmatrix} \bbD_1 & 0 \\ 0 & \bbD_2 \end{pmatrix}-l_i\bbI_p\right) \begin{pmatrix} \bbV_1^*\bbf_i  \\ \bbV_2^*\bbf_i  \end{pmatrix} =0,
	\end{aligned}
	$$
	then we obtian
	\begin{equation}
		\begin{pmatrix} \bbD_1\bbU_1^*\bbS_n\bbU_1\bbD_1-l_i\bbI_M & \bbD_1\bbU_1^*\bbS_n\bbU_2\bbD_2 \\ 
			\bbD_2\bbU_2^*\bbS_n\bbU_1\bbD_1 & \bbD_2\bbU_2^*\bbS_n\bbU_2\bbD_2-l_i\bbI_{p-M} \end{pmatrix}\begin{pmatrix} \pmb\xi_{i1}  \\ \pmb\xi_{i2}  \end{pmatrix}=0,
	\end{equation}
	where $\pmb\xi_{i1}=\bbV_1^*\bbf_i$ and $\pmb\xi_{i2}=\bbV_2^*\bbf_i$.
	
	If we only consider the sample spiked eigenvalues of $\bbB$, $l_i,i \in \mathcal{J}_k$, $k=1,\cdots,K$. As when $n \to \infty$, with probability 1, the limit $\psi(d_k^2)$ of $l_i$ does not belong to interval $ \left[ a_y,b_y\right] $ and the eigenvalues of $\bbD_2\bbU_2^*\bbS_n\bbU_2\bbD_2$ go inside the interval $\left[ a_y,b_y\right]$. 
	Thus, $\bbD_2\bbU_2^*\bbS_n\bbU_2\bbD_2-l_i\bbI_{p-M}$ is invertible with probability 1, we obtain
	\begin{equation}\label{f1}
		\begin{aligned}
			\left( \bbD_1\bbU_1^*\bbS_n\bbU_1\bbD_1-l_i\bbI_M
			-\bbD_1\bbU_1^*\bbS_n\bbU_2\bbD_2\left( \bbD_2\bbU_2^*\bbS_n\bbU_2\bbD_2-l_i\bbI_{p-M}\right) ^{-1}\bbD_2\bbU_2^*\bbS_n\bbU_1\bbD_1\right) \pmb\xi_{i1}=0,
		\end{aligned}
	\end{equation}
	\begin{equation}\label{f1'}
		\left( \bbD_2\bbU_2^*\bbS_n\bbU_2\bbD_2-l_i\bbI_{p-M}\right) \bbu_{i2}+\bbD_2\bbU_2^*\bbS_n\bbU_1\bbD_1\pmb\xi_{i1}=0.
	\end{equation}
	Applying the in-out-exchange formula $\bbX^*\left( \bbX\bbX^*-\lambda\bbI\right) ^{-1}\bbX=\bbI+\lambda\left( \bbX^*\bbX-\lambda\bbI\right) ^{-1}$ to $\dfrac{1}{\sqrt{n}}\bbX^*\bbU_2\bbD_2\cdot\\
	( \bbD_2\bbU_2^*\bbS_n\bbU_2\bbD_2-l_i\bbI_{p-M}) ^{-1}\dfrac{1}{\sqrt{n}}\bbD_2\bbU_2^*\bbX$ in the first equation (\ref{f1}), yields
	\begin{equation}\label{f2}
		\left( \bbI_M+\dfrac{1}{n}\bbD_1\bbU_1^*\bbX\left( n^{-1}\bbX^*\Gamma\bbX-l_i\bbI_n\right) ^{-1}\bbX^*\bbU_1\bbD_1\right) \pmb\xi_{i1}=0,
	\end{equation}
	where $\Gamma=\bbU_2\bbD_2^2\bbU_2^*$.
	
	Define
	\begin{small}
		\begin{equation}\label{Omega}
			\Omega_M\left( \lambda\right) =\dfrac{\lambda}{\sqrt{n}}\left[ \text{tr}\left\lbrace \left( \lambda\bbI_n-\dfrac{1}{n}\bbX^*\Gamma\bbX\right) ^{-1}\right\rbrace \bbI_M-\bbU_1^*\bbX\left( \lambda\bbI_n-\dfrac{1}{n}\bbX^*\Gamma\bbX\right) ^{-1}\bbX^*\bbU_1\right] .
		\end{equation}
	\end{small}
	Set $\Omega_{\psi_k}$ as an $M \times M$ Hermitian matrix, which follows the limiting distribution of $\Omega_M(\psi_n(d_k^2))$ is detailed in Theorem 2 of Jiang and Bai \cite{PG4MT}. 
	By the definition of (\ref{Omega}), it follows from (\ref{f2}) that,
	\begin{equation}\label{f3}
		\left( \bbI_M-\dfrac{1}{n}\text{tr}\left\lbrace \left(l_i \bbI_n-\dfrac{1}{n}\bbX^*\Gamma\bbX\right) ^{-1}\right\rbrace \bbD_1^{2}+\dfrac{1}{l_i\sqrt{n}}\bbD_1 \Omega_M(l_i)\bbD_1 \right) \pmb\xi_{i1}=0.
	\end{equation}
	Because $\left\lbrace l_i/\psi_n(d_k^2)- 1\right\rbrace $ in probability converges to 0 for all $i \in \mathcal{J}_k$, where
	$$\psi_n(d_k^2)=d_k^2\left( 1+c_n\int \frac{t\text{d}H_n(t)}{d_k^2-t}\right) $$
	and $c_n=p/n$ and $H_n$ is the ESD of $\pmb{\Sigma}$, (\ref{f3}) can be written as
	\begin{equation}\label{f4}
		\begin{aligned}
			&\left( \bbI_M-\dfrac{1}{n}\text{tr}\left\lbrace \left( \psi_n(d_k^2)\bbI_n-\dfrac{1}{n}\bbX^*\Gamma\bbX\right) ^{-1}\right\rbrace \bbD_1^{2}\right.\\
			&\ \ \ \ \  \left.+\bbB_1(l_i)+\bbB_2(l_i)+\dfrac{1}{\psi_n(d_k^2)\sqrt{n}}\bbD_1 \Omega_M(\psi_n(d_k^2))\bbD_1 \right) \pmb\xi_{i1}=0,
		\end{aligned}
	\end{equation}
	where
	$$\bbB_1(l_i)=\dfrac{1}{n}\left[ \text{tr}\left( \psi_n(d_k^2)\bbI_n-\dfrac{1}{n}\bbX^*\Gamma\bbX\right)^{-1}-\text{tr}\left( l_i\bbI_n-\dfrac{1}{n}\bbX^*\Gamma\bbX\right)^{-1}\right] \bbD_1^{2},$$ and
	$$\bbB_2(l_i)=\dfrac{1}{l_i\sqrt{n}}\bbD_1 \Omega_M(l_i)\bbD_1-\dfrac{1}{\psi_n(d_k^2)\sqrt{n}}\bbD_1\Omega_M(\psi_n(d_k^2))\bbD_1.$$
	Let $\gamma_{ki}=\sqrt{n}\left( \dfrac{l_i}{\psi_n(d_k^2)}-1\right) $, and by the formula $\bbA^{-1}-\bbB^{-1}=\bbA^{-1}\left( \bbA-\bbB\right) \bbB^{-1}$ with $\bbA,\bbB$ being two arbitrary $n \times n$ invertible matrices, we obtain,

		\begin{align}\label{B1}
			\bbB_1(l_i)&=\left[ \dfrac{1}{n}\text{tr}\left( \psi_n(d_k^2)\bbI_n-\dfrac{1}{n}\bbX^*\Gamma\bbX\right)^{-1}-\dfrac{1}{n}\text{tr}\left( l_i\bbI_n-\dfrac{1}{n}\bbX^*\Gamma\bbX\right)^{-1}\right] \bbD_1^{2}\nonumber\\
			&=\dfrac{1}{n}\text{tr}\left[ \left( \psi_n(d_k^2)\bbI_n-\dfrac{1}{n}\bbX^*\Gamma\bbX\right)^{-1}\left( l_i\bbI_n-\psi_n(d_k^2)\bbI_n\right)
			\left( \psi_n(d_k^2)(1+n^{-1/2}\gamma_{ki})\bbI_n-\dfrac{1}{n}\bbX^*\Gamma\bbX\right)^{-1}\right] \bbD_1^{2}\nonumber\\
			&=\dfrac{\psi_n(d_k^2)\gamma_{ki}}{n^{3/2}}\text{tr}\left[ \left( \psi_n(d_k^2)(1+n^{-1/2}\gamma_{ki})\bbI_n-\dfrac{1}{n}\bbX^*\Gamma\bbX\right)^{-1}\left( \psi_n(d_k^2)\bbI_n-\dfrac{1}{n}\bbX^*\Gamma\bbX\right)^{-1}\right] \bbD_1^{2}\nonumber\\
			&=\left[n^{-3/2}\psi_n(d_k^2)\gamma_{ki}\text{tr}\left( \psi_n(d_k^2)\bbI_n-\dfrac{1}{n}\bbX^*\Gamma\bbX\right)^{-2}+O_p\left( \dfrac{1}{n\psi_n(d_k^2)}\right) \right]\bbD_1^{2}\\
			&= \left[\dfrac{\gamma_{ki}}{\sqrt{n}}\dfrac{\psi_n(d_k^2)}{n} \text{tr}\left( \psi_n(d_k^2)\bbI_n-\dfrac{1}{n}\bbX^*\Gamma\bbX\right)^{-2}\right]\bbD_1^2+O_p\left( \dfrac{1}{n\psi_n(d_k^2)}\right) \bbD_1^{2}\nonumber\\
			&=\dfrac{\gamma_{ki}}{\sqrt{n}}\psi_n(d_k^2)\underline{m}_2(\psi_k)\bbD_1^2+O_p\left( \dfrac{1}{n\psi_n(d_k^2)}\right) \bbD_1^{2}.\nonumber
		\end{align}
	
	Similarly,

		\begin{align}\label{B2}
			\bbB_2(l_i)=&\dfrac{1}{n} \left( \text{tr}\left\lbrace \left( l_i\bbI_n-\dfrac{1}{n}\bbX^*\Gamma\bbX\right) ^{-1}-\left( \psi_n(d_k^2)\bbI_n-\dfrac{1}{n}\bbX^*\Gamma\bbX\right) ^{-1}\right\rbrace \bbD_1^{2}\right)\nonumber\\ 
			&-\dfrac{1}{n}\left( \bbD_1\bbU_1^*\bbX\left( l_i\bbI_n-\dfrac{1}{n}\bbX^*\Gamma\bbX\right) ^{-1}\bbX^*\bbU_1\bbD_1 -\bbD_1\bbU_1^*\bbX\left( \psi_n(d_k^2)\bbI_n-\dfrac{1}{n}\bbX^*\Gamma\bbX\right) ^{-1}\bbX^*\bbU_1\bbD_1\right) \\
			=&O_p\left( \dfrac{1}{n\psi_n(d_k^2)}\right) \bbD_1^{2}.\nonumber
		\end{align}
	
	The calculations are based on the fact that
	$$\dfrac{1}{n}\text{tr}\left( \psi_n(d_k^2)\bbI_n-\dfrac{1}{n}\bbX^*\Gamma\bbX\right)^{-2} \stackrel{a.s.}{\longrightarrow} \underline{m}_2(\psi_k),$$
	and the random vector $\left( \gamma_{ki}, i \in \mathcal{J}_k \right)\textquotesingle$ converges weakly to the joint distribution of the $m_k$ eigenvalues of Gaussian random matrix.
	
	Furthermore, if consider the $k$th diagonal block of the item
	$$\dfrac{1}{n}\text{tr}\left\lbrace \left( \psi_n(d_k^2)\bbI_n-\dfrac{1}{n}\bbX^*\Gamma\bbX\right) ^{-1}\right\rbrace \bbI_M$$
	in (\ref{f4}), we change the related $\underline{m}$ as $\underline{m}_n$, with $H$ substituted by the ESD $H_n$ and $c$ by $c_n$, which satisfies the equation
	$$\psi_n(d_k^2)=-\dfrac{1}{\underline{m}_n(\psi_n(d_k^2))}+c_n\int \frac{t}{1+t\underline{m}_n(\psi_n(d_k^2))}\text{d}H_n(t). $$
	By the formula (4.9) of Theorem 4.2 in Bao et al. \cite{2020baoding}, it can be obtained that
	\begin{equation}\label{th1.1}
		\dfrac{1}{n}\psi_n(d_k^2)\text{tr} \left( \psi_n(d_k^2)\bbI_n-\dfrac{1}{n}\bbX^*\Gamma\bbX\right) ^{-1}+\psi_n(d_k^2)\underline{m}_n(\psi_n(d_k^2))=O_p\left( \dfrac{1}{n}\right).
	\end{equation}
	Note that $\psi_n(d_k^2)$ is the inverse of the Stieltjes transform $\underline{m}_n(\psi_n(d_k^2))$ at $-1/d_k^2$, we have $\underline{m}_n(\psi_n(d_k^2))=-1/d_k^2$.
	
	Recall (\ref{f4}). For every sample spiked eigenvalue, $l_i$, $j\in\mathcal{J}_k$, $k=1,\cdots,K$, we can get
	\begin{equation}\label{f5}
		\begin{aligned}
			0&=\left( \bbI_M+\underline{m}_n(\psi_n(d_k^2))\bbD_1^2+\dfrac{\gamma_{ki}}{\sqrt{n}}(\psi_n(d_k^2)\underline{m}_2(\psi_n(d_k^2))) \bbD_1^2\right.\\
			&\ \ \ \ \ \ \ \  \left.+\dfrac{1}{\psi_n(d_k^2)\sqrt{n}}\bbD_1\Omega_M(\psi_n(d_k^2))\bbD_1+O_p\left( \dfrac{1}{n\psi_n(d_k^2)}\right)\bbD_1^2 \right) \pmb\xi_{i1}\\
			&=\left( \psi_n(d_k^2)\bbI_M+\psi_n(d_k^2)\underline{m}_n(\psi_n(d_k^2))\bbD_1^2
			+\dfrac{1}{\sqrt{n}}\bbD_1\Omega_M(\psi_n(d_k^2))\bbD_1\right.\\
			&\ \ \ \ \ \ \ \ \ \ \ \ \ \ \ \ \ \ \ \ \ \ \ \ \ \ \ \ \ \ \ \ \ \ \ \ \ \  \left.+\dfrac{\gamma_{ki}}{\sqrt{n}}(\psi_n^2(d_k^2)\underline{m}_2(\psi_n(d_k^2))) \bbD_1^2+o_p\left( \dfrac{1}{\sqrt{n}}\right)\bbD_1^2 \right) \pmb\xi_{i1}
		\end{aligned}
	\end{equation}
	by the equations (\ref{B1}), (\ref{B2}) and (\ref{th1.1}). 
	
	More specifically, by (\ref{f5}) and applying Skorokhod strong representation theorem (see Skorohod \cite{1956Skorohod}, Hu and Bai \cite{2014hubai}), by selecting an appropriate probability space, it is possible to suggest that the convergence of $\Omega_M(\psi_n(d_k^2))$ and (\ref{f5}) conforms to the `almost surely' criterion (see Jiang and Bai \cite{2021jiangbai}), it produces

			\begin{align}\label{f6}
				&\left(\begin{pmatrix} \psi_n(d_k^2)\left(1- \dfrac{d_1^2}{d_k^2}\right) \bbI_{m_1} & 0 & \cdots & 0\\
					0 & \ddots &  &  \\
					&  &  \psi_n(d_k^2)\left( 1+d_k^2\underline{m}_n(\psi_n(d_k^2))\right) \bbI_{m_k} & \vdots\\
					\vdots &  &  &  \\
					&  & \ddots & 0\\
					0 & \cdots & 0 &\psi_n(d_k^2)\left( 1-\dfrac{d_K^2}{d_k^2}\right) \bbI_{m_K}
				\end{pmatrix} \right.\\
				&\left.+\dfrac{\gamma_{ki}}{\sqrt{n}}\begin{pmatrix} d_1^2\psi_n^2(d_k^2)\underline{m}_2(\psi_n(d_k^2))\bbI_{m_1} &0  &\cdots  & 0 \\
					0	& \ddots &  &  \\
					&  & d_k^2\psi_n^2(d_k^2)\underline{m}_2(\psi_n(d_k^2))\bbI_{m_k} &\vdots \\
					&  & \ddots & 0 \\
					0	&  \cdots& 0 & d_K^2\psi_n^2(d_k^2)\underline{m}_2(\psi_n(d_k^2))\bbI_{m_K}
				\end{pmatrix}\right.\nonumber\\
				&\left.+\dfrac{1}{\sqrt{n}}\bbD_1\Omega_M(\psi_n(d_k^2))\bbD_1+o\left(\dfrac{1}{\sqrt{n}} \right)\bbD_1^2  \right)\pmb\xi_{i1}=0,\nonumber
	\end{align}
	with noting $\underline{m}_n(\psi_n(d_k^2))=-1/d_k^2$.
	
	Splitting $\pmb\xi_{i1}$ into $K$ subvectors with dimensions $m_k,k=1,\cdots,K$. For the population eigenvalues $d_u^2$ in the $u$th diagonal block of $\bbD_1^2$, if $u\neq k$, then $\underline{m}_n(\psi_n(d_k^2))=-\dfrac{1}{d_k^2}$ by the definition of $\psi_n(d_k^2)$, hence $\psi_n(d_k^2)\left(1- \dfrac{d_1^2}{d_k^2}\right)$ keeps away from 0, by the separation condition of spikes (\ref{sc}). Moreover, $\psi_n(d_k^2)( 1+d_k^2\underline{m}_n(\psi_n(d_k^2)))=0$ by definition. Then, from (\ref{f6}), one can show that
	\begin{equation}
		\pmb\xi_{i1u} \to 0, \ \ \text{for}\  u \neq k,
	\end{equation}
	and with multiplying $n^{1/4}$ to the $k$th block row and $k$th block column of (\ref{f6}),
	\begin{equation}\label{f7}
		d_k^2\left( \gamma_{ki}\psi_n^2(d_k^2)m_2(\psi_n(d_k^2))\bbI_{m_k}+\Omega_{kk}+o(1)\right) \pmb\xi_{i1k}=0
	\end{equation}
	where $\Omega_{kk}$ is the $k$-th diagonal block of $\Omega_{\psi_k}$. 
	
	From (\ref{f7}), one conclude that $\left( \gamma_{ki}d_k^2\psi_n^2(d_k^2)m_2(\psi_n(d_k^2))\bbI_{m_k},i \in \mathcal{J}_k \right) ^\prime $ tends to the $m_k$ eigenvalues of the $m_k \times m_k$ matrix $-d_k^2\Omega_{kk}$, where $-\Omega_{kk}$ is the limit of $\Omega_M(\psi_n(d_k^2))$ in Corollary 3.1 of Jiang and Bai \cite{2021jiangbai}, and $\pmb\xi_{i1k}$ tends to an correspongding eigenvector (dimension $m_k$).

	By (\ref{f1'}), we have
	\begin{equation}\label{u2}
		\begin{aligned}
			\left\| \pmb\xi_{i2}\right\| ^2=&\left\langle \pmb\xi_{i2},\pmb\xi_{i2}\right\rangle \\
			=&\pmb\xi_{i2}^*\pmb\xi_{i2}\\
			=&\pmb\xi_{i1}^*\bbD_1\bbU_1^*\bbS_n\bbU_2\bbD_2\left( \bbD_2\bbU_2^*\bbS_n\bbU_2\bbD_2-l_i\bbI_{p-M}\right) ^{-2}\bbD_2\bbU_2^*\bbS_n\bbU_1\bbD_1\pmb\xi_{i1}.
		\end{aligned}	
	\end{equation}

	Combining $\pmb\xi_{i1u} \to 0$ for all $u\neq k$, (\ref{u2}) becomes
	\begin{small}\begin{equation}\label{2}
			\begin{aligned}
				&\left\| \pmb\xi_{i2}\right\| ^2\\=&\pmb\xi_{i1k}^*d_k^2\bbU_{1k}^*\bbS_n\bbU_2\bbD_2\left( \bbD_2\bbU_2^*\bbS_n\bbU_2\bbD_2-l_i\bbI_{p-M}\right) ^{-2}\bbD_2\bbU_2^*\bbS_n\bbU_{1k}\pmb\xi_{i1k}+o(1)\\
				=&\dfrac{1}{n}\pmb\xi_{i1k}^*d_k^2\bbU_{1k}^*\bbX\bbX^*\bbU_2\bbD_2\left( \dfrac{1}{n}\bbD_2\bbU_2^*\bbX\bbX^*\bbU_2\bbD_2-l_i\bbI_{p-M}\right) ^{-2}\dfrac{1}{n}\bbD_2\bbU_2^*\bbX\bbX^*\bbU_{1k}\pmb\xi_{i1k}+o(1).
			\end{aligned}	
	\end{equation}\end{small}
	
	Recall (\ref{2}). Similar to Jiang and Bai \cite{PG4MT}, using Lemma \ref{PG4MT}, one may change $\bbU_2^*\bbX$ as $\bbZ_{(p-M)\times n}$ of i.i.d. $N(0,1)$ entries. Then, we have
	
		\begin{align}
			&\left\| \pmb\xi_{i2}\right\| ^2\nonumber\\
			=&\dfrac{1}{n^2}\pmb\xi_{i1k}^*d_k^2\bbY_k\bbZ^*\bbD_2\left( \dfrac{1}{n}\bbD_2\bbZ\bbZ^*\bbD_2-l_i\bbI_{p-M}\right) ^{-2}\bbD_2\bbZ\bbY_k^*\pmb\xi_{i1k}+o(1)\nonumber\\
			=&\dfrac{1}{n^2}\left\| \pmb\xi_{i1k}\right\| ^2\dfrac{d_k^2}{m_k} {\rm tr}E\left[ \bbY_k\bbZ^*\bbD_2\left( \dfrac{1}{n}\bbD_2\bbZ\bbZ^*\bbD_2-l_i\bbI_{p-M}\right) ^{-2}\bbD_2\bbZ\bbY_k^*\right] +o_{a.s.}(1)\nonumber\\
			=&\dfrac{1}{n^2}\left\| \pmb\xi_{i1k}\right\| ^2d_k^2E\left[ \dfrac{1}{m_k} {\rm tr}\left( \bbY_k\bbZ^*\bbD_2\left( \dfrac{1}{n}\bbD_2\bbZ\bbZ^*\bbD_2-l_i\bbI_{p-M}\right) ^{-2}\bbD_2\bbZ\bbY_k^*\right) \right.\\
			&\left.\ \ \ \ \ \ \ \ \ \ \ \ \ \ \ \ \ \ \ \ \ \ \ \ \ \ \ \ -{\rm tr}\left( \bbZ^*\bbD_2\left( \dfrac{1}{n}\bbD_2\bbZ\bbZ^*\bbD_2-l_i\bbI_{p-M}\right) ^{-2}\bbD_2\bbZ\right) \right.\nonumber\\
			& \left.\ \ \ \ \ \ \ \ \ \ \ \ \ \ \ \ \ \ \ \ \ \ \ \ \ \ \ \ +{\rm tr}\left( \bbZ^*\bbD_2\left( \dfrac{1}{n}\bbD_2\bbZ\bbZ^*\bbD_2-l_i\bbI_{p-M}\right) ^{-2}\bbD_2\bbZ\right) -\left(\underline{m}_{n1}(\psi_{nk})+\psi_{nk}\underline{m}_{n2}(\psi_{nk})\right)\right.\nonumber\\
			& \left.\ \ \ \ \ \ \ \ \ \ \ \ \ \ \ \ \ \ \ \ \ \ \ \ \ \ \ \ +\left(\underline{m}_{n1}(\psi_{nk})+\psi_{nk}\underline{m}_{n2}(\psi_{nk})\right)	\right] +o_{a.s.}(1).\nonumber
		\end{align}

	Moreover, by Bai and Silverstein \cite{2004bai}, one knows that
	\begin{small}\begin{equation}\label{04}
		\begin{aligned}
			&d_k^2\left( n^{-2}{\rm tr}\left( \bbZ^*\bbD_2\left( n^{-1}\bbD_2\bbZ\bbZ^*\bbD_2-l_i\bbI_{p-M}\right) ^{-2}\bbD_2\bbZ\right)-\left(\underline{m}_{n1}(\psi_{nk})+\psi_{nk}\underline{m}_{n2}(\psi_{nk})\right)\right) =O_p\left( \dfrac{1}{\sqrt{n}}\right) .
		\end{aligned}
	\end{equation} \end{small}
	
	Let $L=\dfrac{d_k^2}{n^2m_k}{\rm tr}\left( \bbY_k\bbZ^*\bbD_2\left( n^{-1}\bbD_2\bbZ\bbZ^*\bbD_2-l_i\bbI_{p-M}\right) ^{-2}\bbD_2\bbZ\bbY_k^*\right)$, it follows from Lemma \ref{cor} and (\ref{04}) that $\sqrt{n}\left( L-d_k^2\left(\underline{m}_{n1}(\psi_{nk})+\psi_{nk}\underline{m}_{n2}(\psi_{nk})\right)\right) $
	tends to a Gaussian variable whose mean is zero and variance  $\theta_k=d_k^4\left( \dfrac{2}{m_k}\sigma_k+\dfrac{\kappa_{x,k}}{m_k^2}\rho_k^2\right) $.
	Therefore, from the second equation in (\ref{2}), by using the Delta Method, Theorem \ref{th2} can be concluded.	
\end{proof}
\subsection{Proof of Theorem \ref{th1}}

Next, let us estimate the asymptotic length of $\bbu_{i1k}$. According to the conclusion of Theorem \ref{th2}, we have
$$\left\| \pmb\xi_{i1k}\right\| ^2 \xrightarrow{D}\dfrac{1}{1+d_k^2\left( \underline{m}_1(\psi_k)+\psi_k \underline{m}_2(\psi_k)\right)}.$$
Since $\dfrac{1}{1+d_k^2\left( \underline{m}_1(\psi_k)+\psi_k \underline{m}_2(\psi_k)\right)}$ is a constant, we can obtain

\begin{equation}
	\left\| \pmb\xi_{i1k}\right\| ^2\xrightarrow{P}\dfrac{1}{1+d_k^2\left( \underline{m}_1(\psi_k)+\psi_k \underline{m}_2(\psi_k)\right)}.
\end{equation}

\subsection{Proofs of Theorem \ref{th3} and Corollary \ref{corollary}}\label{test1}

Let $R_{11}(\psi_1)$ be a Gaussian random variable with mean 0 and $$Var(R_{11}(\psi_1))=(2\theta+\beta_1\omega)d_1^4, \ \text{with} \ \beta_1=d_1^4Ex_{11}^4-3,$$
where
$$\theta=1+2cm_1(\psi_1)+cm_1(\psi_1),\ \ \ \ \ \omega=1+2cm_1(\psi_1)+\left( \dfrac{c(1+m_1(\psi_1))}{\psi_1-c(1+m_1(\psi))}\right) ^2,$$
with
$$m_1(\psi_1)=\int\dfrac{x}{\psi_1-x}\text{d}F(x),\ \ \ \ \ m_2(\psi_1)=\int\dfrac{x^2}{(\psi_1-x)^2}\text{d}F(x).$$
Note that $F$ is the LSD of $\bbD_2\bbU_2^*\bbX\bbX^*\bbU_2\bbD_2/n$. For example, assuming that the variables are real Gaussian, we remark that $$Var(R_{11}(\psi_1))=2\theta d_1^4,$$
where $$\theta=1+2cm_1(\psi_1)+cm_1(\psi_1)=\dfrac{(d_1^2-1+c)^2}{(d_1^2-1)^2-c}.$$

Then, we define
$$G(\psi_1)=\dfrac{1}{1+cm_3(\psi_1)d_1^2}R_{11}(\psi_1),\ \tilde{G}(\psi_1)=Var(G(\psi_1)),$$
where $$m_3(\psi_1)=\int\dfrac{x}{(\psi_1-x)^2}\text{d}F(x).$$
We define as $$\tau(d_1^2)=(\nu\textquotesingle(d_1^2))^2\left( \dfrac{1-c}{\psi_1^2}+c\int \dfrac{\text{d}F(x)}{(x-\psi_1)^2}\right) .$$
\begin{proof}[Proof of Theorem \ref{th3}]
	By definition of $\mathcal{T}_1$ in (\ref{TT1}), we have the following derivation
	\begin{equation}
		\begin{aligned}
			\sqrt{n}\mathcal{T}_1&=\dfrac{\sqrt{n}\left( \langle\bbv_1,\bbf_1\rangle^2-\nu(\widehat{d_1^2})\right)}{d_1^2}\\
			&=\dfrac{\sqrt{n}\left( \langle\bbv_1,\bbf_1\rangle^2-\nu(d_1^2)+\nu(d_1^2)-\nu(\widehat{d_1^2})\right)}{d_1^2}\\
			&=\dfrac{\sqrt{n}\left( \langle\bbv_1,\bbf_1\rangle^2-\nu(d_1^2)\right) }{d_1^2}-\dfrac{\sqrt{n}\left( \nu(\widehat{d_1^2} )-\nu(d_1^2)\right)}{d_1^2}\\
			&=\dfrac{\sqrt{n}\left( \langle\bbv_1,\bbf_1\rangle^2-\nu(d_1^2)\right) }{d_1^2}-\dfrac{\nu\textquotesingle(d_1^2)}{d_1^2}\sqrt{n}(\widehat{d_1^2}-d_1^2),
		\end{aligned}
	\end{equation}
	Note that the asymptotic distribution of $\mathcal{T}_1$ is determined by $\widehat{d_1^2}-d_1^2$, which follows from Theorem 3.2 of \cite{2012dingbai}, we konw that $\sqrt{n}\mathcal{T}_1$ have the same limiting distribution as $\sqrt{n}(\widehat{d_1^2}-d_1^2)$ up to a multiplicative constant $\dfrac{\nu\textquotesingle(d_1^2)}{d_1^2}$. Thus the Theorem \ref{th3} can be concluded. 
\end{proof}

\begin{proof}[Proof of Corollary \ref{corollary}]
	The conclusion in Corollary \ref{corollary} follows from Theorem \ref{th3} and Theorem 4.1 in \cite{2012baiyao} and the fact that $\widehat{d_1^2}$ is a consistent estimator of $d_1^2$ after appropriate scaling.
\end{proof}

\section{Conclusion}\label{conclusion}

In this paper, we propose asymptotic first-order and distributional results for eigenvectors of a generalized spiked model. Unlike Bao et al. \cite{2020baoding}, we replace the assumption of infinite moments with a tail probability $\lim\limits_{\tau\to\infty}\tau^4P(|x_{ij}|>\tau)=0$. This condition serves as both a necessary and sufficient condition for edge universality at the largest eigenvalue. Furthermore, we eliminate the requirement for the form of the spiked covariance matrix, $\pmb{\Sigma}=\bbV\bbD^2\bbV^*$, in favor of $\pmb{\Sigma}=\bbI+\bbS$. Without the constraint of the existence of the 4th moment, we only need a more common and minor condition (\ref{ux}). Additionally, we establish a Central Limit Theorem (CLT) for the eigenvectors corresponding to the spiked eigenvalues outside the interval $[1-\sqrt{c},1+\sqrt{c}]$.

\begin{acks}[Acknowledgments]
	
	Jiang Hu was partially supported by NSFC Grants No. 12171078, No. 12292980, No. 12292982, National Key R \& D Program of China No. 2020YFA0714102, and Fundamental Research Funds for the Central Universities, China No. 2412023YQ003. 
	
	Zhidong Bai was partially supported by NSFC Grants No. 12171198, No. 12271536, and Team Project of Jilin Provincial Department of Science and Technology No. 20210101147JC.

\end{acks}

\bibliographystyle{imsart-number.bst}
\bibliography{reference.bib}

\end{document}